\documentclass[a4paper,12pt]{article}
\usepackage[top=2.5cm,bottom=2.5cm,left=2.5cm,right=2.5cm]{geometry}
\usepackage{cite, amsmath, amssymb}
\usepackage{amssymb}
\usepackage{graphicx}
\newenvironment{proof}{{\noindent \it Proof.}}{\hfill $\blacksquare$\par}
\usepackage{latexsym}
\usepackage{caption}
\usepackage{graphicx,booktabs,multirow}
\usepackage{tikz}
\usepackage{bigdelim}
\usepackage{nicematrix}
\usepackage{mathtools}
\usepackage{makecell}
\usetikzlibrary{decorations.pathreplacing}
\usetikzlibrary{intersections}
\usepackage{enumerate}
\usepackage{graphicx,booktabs,multirow}
\usepackage{appendix}
\newtheorem{theorem}{Theorem}[section]
\newtheorem{proposition}[theorem]{\rm\bfseries Proposition}
\newtheorem{lemma}[theorem]{Lemma}

\newtheorem{remark}[theorem]{Remark}

\newtheorem{definition}{Definition}[section]
\usepackage[section]{placeins}
\def\NAT@def@citea{\def\@citea{\NAT@separator}}% Suppress spaces between citations using natbib.sty
\allowdisplaybreaks

\begin{document}
\vspace*{10mm}

\noindent
{\Large \bf  Connected graphs with large multiplicity of $-1$ in the spectrum of the eccentricity matrix}

\vspace*{7mm}

\noindent
{\large \bf Xinghui Zhao, Lihua You*}
\noindent

\vspace{7mm}

\noindent
 School of Mathematical Sciences, South China Normal University,  Guangzhou, 510631, P. R. China,
e-mail: {\tt 2022021990@m.scnu.edu.cn},\quad{\tt ylhua@scnu.edu.cn}.\\[2mm]
$^*$ Corresponding author
\noindent

%\footnotesize $^1${\it School of Mathematical Sciences, South China Normal University, Guangzhou, 510631, P. R. China}\\
%\footnotesize $^2${\it Department of Mathematics Teaching, Guangzhou Civil Aviation College, Guangzhou, 510403, P. R. China}\\
%\noindent
% $^2${\it Department of Mathematics Teaching, Guangzhou Civil Aviation College, Guangzhou, 510403, P. R. China\/} \\
\vspace{7mm}

\noindent
{\bf Abstract} \ 
\noindent
The eccentricity matrix of a simple connected graph is obtained from the distance matrix by only keeping the largest distances for each row and each column, whereas the remaining entries become zero. This matrix is also called the anti-adjacency matrix, since the adjacency matrix  can also be obtained from the distance matrix but this time by keeping only the entries equal to $1$. It is known that, for $\lambda \not\in \{-1,0\}$ and a fixed $i\in \mathbb{N}$, there is only a finite number of graphs with $n$ vertices having $\lambda$ as an eigenvalue of multiplicity $n-i$ on the spectrum of the adjacency matrix. This phenomenon motivates researchers to consider the graphs has a large multiplicity of an eigenvalue in the spectrum of the eccentricity matrix, for example, the eigenvalue $-2$ [X. Gao, Z. Stani\'{c}, J.F. Wang, Grahps with large multiplicity of $-2$ in the spectrum of the eccentricity matrix,  Discrete Mathematics, 347 (2024) 114038]. In this paper, we characterize the connected graphs with $n$ vertices having $-1$ as an eigenvalue of multiplicity $n-i$ $(i\leq5)$ in the spectrum of the eccentricity matrix. Our results also become meaningful in the framework of
the median eigenvalue problem [B. Mohar, Median eigenvalues and the HOMO-LUMO index of graphs, Journal of Combinatorial Theory Series B, 112 (2015) 78-92].
 \\[2mm]
%\vspace{5mm}


\baselineskip=0.30in

\section{Introduction}

\hspace{1.5em}Throughout this paper we consider the graphs as simple and connected graphs. Let $G=(V(G),E(G))$ be a simple  graph with vertex set $V(G)=\{v_1,\cdots,v_n\}$ and edge set $E(G)$. If the vertices $v_i$ and $v_j$ are adjacent, we write  $v_i\sim v_j$ and $v_iv_j \not\sim E(G)$ otherwise. The set of the neighborhood of $v_i$ in $G$ is denoted by $N_{G}(v_i)$ (for short $N(v_i)$). Then the degree of the vertex $v$ is equal $|N(v)|$ (the cardinality of the set $N(v)$), denoted by $d_G(v)$, and $\bigtriangleup(G)=\max\{d_G(v)|v\in V(G)\}$. As usual, $K_n, C_n, P_n$ is the complete graph, cycle, path with order $n$, respectively. Let $G\cup H$ be the disjoint union of graphs $G$ and $H$. The join of two disjoint graphs $G$ and $H$, denoted by $G\vee H$, is the graph obtained by joining each vertex of $G$ to each vertex of $H$. The complement of graph $G$ is denoted by $\overline{G}$.

The distance $d_{G}(u,v)$ (for short $d(u,v)$) between two vertices $u$ and $v$ is the minimum length of the paths joining them. The diameter of $G$, denoted by
$diam(G)$, is the greatest distance between any two vertices in $G$. The eccentricity $\varepsilon_{G}(v)$ (for short $\varepsilon(v)$) of the vertex $v\in V(G)$ is given by
$\varepsilon_{G}(v) = \max\{d(u, v) | u\in V(G)\}$. Then the eccentricity matrix $\mathcal{A}(G) = (\epsilon_{uv}(G))$ of $G$ is defined as follows:
$$\epsilon_{uv}(G)=\begin{cases}
	d(u,v),&\text{if} \ d(u,v)=\min\{\varepsilon_{G}(u), \varepsilon_{G}(v)\},\\
	0,& \ \text{others}.
\end{cases}$$

Let $M$ be an $n\times n$ matrix, $P(M,\lambda)$ be characteristic polynomial of $M$,  $\xi_1(M)\geq \xi_2(M)\geq \cdots\geq \xi_n(M)$ be roots of $P(M,\lambda)=0$. We also call $\xi_1(M),\xi_2(M),\cdots, \xi_n(M)$ the eigenvalues of $M$. The spectrum of $M$ is a multiset of all eigenvalues of $M$, denoted by $Spec(M)$. The multiplicity of an eigenvalue $\xi$ in the $Spec(M)$ is denoted by $m_{M}(\xi)$.

 Since the eccentricity matrix $\mathcal{A}(G)$ is a real symmetric matrix, all its
 eigenvalues, called $\mathcal{A}$-eigenvalues of $G$, are real. Then we can arrange the $\mathcal{A}$-eigenvalues of $G$ as $\xi_1\geq \xi_2\geq \dots \geq \xi_n$. The $\mathcal{A}$-spectrum of $G$ is a multiset consisting of the $\mathcal{A}$-eigenvalues, denoted by $Spec_{\mathcal{A}}(G)$. If $\xi^{'}_1> \xi^{'}_2> \dots > \xi^{'}_k$ are the pairwise distinct $\mathcal{A}$-eigenvalues of $G$, then the $\mathcal{A}$-spectrum of $G$ can be written as
\begin{equation*}
	Spec_{\mathcal{A}}(G)=\{\xi_1, \xi_2,\dots, \xi_n\}= \left\{
	\begin{array}{cccc}
		\xi^{'}_1 &\xi^{'}_2&\dots &\xi^{'}_k \\
		m_1 &m_2 &\dots &m_k
	\end{array}
	\right\},
\end{equation*}
where $m_i=m_{\mathcal{A}(G)}(\xi^{'}_i)$ is the algebraic multiplicity of the eigenvalue $\xi^{'}_i
(1 \leq i \leq k)$. 
%Specially, the multiplicity of an eigenvalue $\xi$ in the $\mathcal{A}$-spectrum of $G$ is also denoted by $m_{G}(\xi)$ (for short $m(\xi)$).

 Other symbols not introduced in this paper can refer to \cite{bd} and \cite{dc}.
 
  Clearly, the matrix $\mathcal{A}(G)$ is obtained from the distance matrix by only keeping the largest distances for each row and each
 column, whereas the remaining entries become $0$. That is why $\mathcal{A}(G)$ can be interpreted as the  anti-adjacency
 matrix, where the adjacency matrix obtained from the distance matrix by keeping only distances equal to 1 on each row and each
 column. From this point of view, $\mathcal{A}(G)$ and $A(G)$ are extremal among all possible distance-like matrices. As a contrast
 with $A(G)$, the eccentricity matrix has some fantastic properties, one of which is that $\mathcal{A}(G)$ of a connected graph is not
 necessarily irreducible, see the published \cite{wang2,ss,zpq}.

 In the last few decades, the eccentricity matrices of graphs inspired much interest and attracted the attention of many researchers, which caused more than forty papers published. For example, Wang et al. \cite{wang2} showed that the eccentricity matrix of trees is irreducible, Wei, He and Li \cite{wei}  determined the $n$-vertex trees with minimum $\mathcal{A}$-spectral radius and all trees with given order and diameter having minimizing  $\mathcal{A}$-spectral radius, which solved two conjectures in \cite{wang2}, 
 Wang et al. \cite{wang} studied the $\mathcal{A}$-spectral radius properties of graphs, Lei, Wang and Li \cite{lei} characterized the graphs whose second least $\mathcal{A}$-eigenvalue is greater than $-\sqrt{15-\sqrt{193}}$, Mahato and Kannan \cite{mahat} studied the inertia and spectral symmetry of the eccentricity matrix of trees,  Wang et al. \cite{wang3} characterized
 all connected graphs with exactly one positive $\mathcal{A}$-eigenvalue and so on.
 
 Here is the motivation for the research reported in this paper. An experienced reader knows that the eigenvalues $0$ and $-1$ can attain a very large multiplicity in the spectrum of the adjacency matrix of some connected graphs. For example, the multiplicity of $-1$ in the spectrum of a complete graph $K_n$ is $n-1$, while the multiplicity of $0$ in the spectrum of a complete bipartite graph (of order $n$) is $n-2$. Recently, Gao, Stani\'{c} and Wang \cite{gx} characterized graphs with $m_{\mathcal{A}(G)}(-2)= n-i$ with $i\leq 5$. This motivates us to consider the graphs for which a particular eigenvalue has a large multiplicity in the spectrum of the eccentricity matrix.

 In this paper, we will characterize graphs with $m_{\mathcal{A}(G)}(-1)= n-i$ with $i\leq 5$, and obtain the following results, which will become meaningful in the framework of
 the median eigenvalue problem \cite{mo,mo2} since, for sufficiently large $n$, the median eigenvalue of the
 obtained graphs is always $-1$.
 
 For convenience, we use $m(\xi)$ to denoted $m_{\mathcal{A}(G)}(\xi)$ in the following.
 \begin{theorem}\label{1}
 	Let $G$ be a connected graph with order $n$, $\mathcal{G}_1=\{K_{n-4}\vee4K_1,K_{n-4}\vee(2K_1\cup K_2),K_{n-4}\vee(P_3\cup K_1),K_{n-4}\vee2K_2,K_{n-4}\vee P_4,K_{n-4}\vee(K_3\cup K_1),K_{n-4}\vee C_4\}$. Then we have
 	
 	{\rm \item(i)} $m(-1)= n-1$ if and only if $G\cong K_n$;
 	
 	{\rm \item(ii)} $m(-1)\neq n-2$;
 	
 	{\rm \item(iii)} $m(-1)= n-3$ for $n\geq 4$ if and only if $G\in\{ K_{n-2}\vee2K_1,P_4\} $;
 	
 	{\rm \item(iv)} $m(-1)= n-4$ for $n\geq 9$ if and only if $G\in\{  K_{n-3}\vee(K_2\cup K_1),  K_{n-3}\vee 3K_1 \}$;
 	
 	{\rm \item(v)} $m(-1)= n-5$ for $n\geq 16$ if and only if $G\in \mathcal{G}_1\cup\{K_{n-5}\vee C_5,K_{n-5}\vee(K_1\cup P_4),K_{n-5}\vee H_1\}$, where $H_1$ is illustrated in Figure $1$.
 	
 \end{theorem}

  \begin{figure}[!h]
 	\centering
 	\begin{tikzpicture}
 		\node[anchor=south west,inner sep=0] (image) at (0,0) {\includegraphics[width=0.2\textwidth]{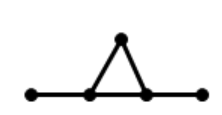}};
 		\begin{scope}[
 			x={(image.south east)},
 			y={(image.north west)}
 			]
 			\node [black, font=\bfseries] at (0.53,-0.15) { };
 	
 		\end{scope}
 	\end{tikzpicture}
 	
 	\caption{ $H_1$.}\label{Figure 1}
 \end{figure}

\section{Preliminaries}\label{sec-pre}
\hspace{1.5em} In this section, we recall some definitions and results which will be used in next section.

\begin{definition}{\rm(\!\!\cite{AE})}\label{d1}
	Let $M$ be a complex matrix of order $n$ described in the following
	block form\begin{equation*}
		M=\begin{pmatrix}
			M_{11}&\cdots&M_{1l}\\
			\vdots&\ddots&\vdots\\
			M_{l1}&\cdots&M_{ll}
		\end{pmatrix}
	\end{equation*}
	where the blocks $M_{ij}$ are $n_i \times n_j$ matrices for any $1\leq i, j \leq l$ and $n= n_1 +\cdots + n_l$.
	For $1 \leq i, j \leq l$, let $q_{ij}$ denote the average row sum of $M_{ij}$, i.e. $q_{ij}$ is the sum of all
	entries in $M_{ij}$ divided by the number of rows. Then $Q(M)=(q_{ij})$ (or simply $Q$) is
	called the quotient matrix of $M$. If, in addition, for each pair $i, j$, $M_{ij}$ has a constant row
	sum, i.e., $M_{ij}\vec{e}_{n_j}= q_{ij}\vec{e}_{n_i}$, then $Q$ is called the equitable quotient matrix of $M$, where
	$\vec{e}_k=(1, 1,\cdots, 1)^T \in \mathcal{C}^k$, and $\mathcal{C}$ denotes the field of complex numbers.
\end{definition}

\begin{lemma}{\rm(\!\!\cite{ylh})}\label{qjz}
	Let $M$ be defined as Definition \ref{d1} such that $M_{ij}=s_{ij}J_{n_i,n_j}$ for $i\neq j$, and $M_{ii} =
	s_{ii}J_{n_i,n_i} + p_iI_{n_i}$. Then the equitable quotient matrix of $M$ is $Q=(q_{ij})$ with $q_{ij}=s_{ij}n_j$
	if $i \neq j$, and $q_{ii} = s_{ii}n_i + p_i$ for $1\leq i, j \leq l$. Moreover,\begin{equation*}
		Spec(M)=Spec(Q)\cup \left\{
		\begin{array}{cccc}
			p_1 &\cdots & p_l\\
			n_1-1 &\cdots &n_l-1
		\end{array}
		\right\}.
	\end{equation*} 
	
\end{lemma}
 
%Let $P_n$, $\overline{G}$, $G_1\vee G_2$ denote path of order $n$, \emph{complement} of graph $G$, the graph $G=(V,E)$ obtained from two disjoint graphs $G_1, G_2$
%by joining each vertex of $G_1$ to each vertex of $G_2$ with vertex set $V = V(G_1)\cup V(G_2)$
%and edge set $E = E(G_1)\cup E(G_2)\cup \{uv|u\in V(G_1), v\in V (G_2)\}$, respectively.

%\begin{lemma}{\rm(\!\!\cite{ylh})}\label{shang}
 %Let $M$ be a complex matrix with the equitable quotient
%matrix $Q(M)$. Then $\sigma(Q(M)) \subset \sigma(M)$, where $\sigma(M)$ denote the spectrum of $M$ including algebraic multiplicityis characteristic polynomial.
%\end{lemma}

\begin{lemma}{\rm(\!\!\cite{is})} \label{pri}
	Let $M$ be an $ n\times n$ Hermitian matrix, $M^*$ be its $k\times k$  principal submatrix, and $\xi_i(M)$ be the $i$-th largest eigenvalue of $M$. Then  for $i \in\{1, 2, . . ., k\}$, we have $$\xi_i(M) \geq \xi_i(M^* ) \geq \xi_{n-k+i}(M).$$
\end{lemma}

A graph is $H$-free if it does not contain $H$ as an  induced subgraph, where $H$ can be a graph or a set of graphs. 

\begin{lemma}{\rm(\!\!\cite{vz})}\label{fr}
	A graph is a complete multipartite graph if and only if it is $\overline{P_3}$-free.
\end{lemma}

\begin{lemma}{\rm(\!\!\cite{hl})}\label{fr2}
	Let $G$ be a connected graph with order $n\geq 3$. Then $G\cong K_l\vee (K_s\cup K_t)$ if and only if $G$ is $\{K_{1,3},P_4,C_4\}$-free, where $s,t\geq 0$ and $l\geq 1$. 
\end{lemma}

%\begin{lemma}{\rm(\!\!\cite{wang3})} \label{w}
%	Let $G$ be a connected graph and $H$ be an induced subgraph of graph $G$. For any vertices $u, v \in V(H)$, if $\varepsilon_{H}(v)=\varepsilon_{G}(v)$ and $d_H (u, v) = d_G(u, v)$, then
%	$\mathcal{A}(H)$ is a principal submatrix of $\mathcal{A}(G)$.
%\end{lemma}

\begin{lemma}{\rm(\!\!\cite{gx})} \label{gx}
	The set $\{v_1,v_2,\cdots,v_k\}$ of either duplicate or co-duplicate vertices of a graph $G$ gives rise to an eigenvalue $\xi$ of multiplicity at least $k-1$ in the spectrum of $\mathcal{A}(G)$, along with the following:
	$$\xi=\begin{cases}
		-2,&\text{if}\ \varepsilon(v_1)=2\ \text{and}\ v_1,\cdots,v_k\ \text{are duplicate}; \\
		0,&\text{if}\ \varepsilon(v_1)\neq2\ \text{and}\ v_1,\cdots,v_k\ \text{are duplicate}; \\
		-1,&\text{if}\ \varepsilon(v_1)=1\ \text{and}\ v_1,\cdots,v_k\ \text{are co-duplicate};\\
		0,&\text{if}\ \varepsilon(v_1)\neq1\ \text{and}\ v_1,\cdots,v_k\ \text{are co-duplicate}.
	\end{cases}$$
\end{lemma}
 
 Let $G$ be a graph with vertex set $V(G)=\{v_1,\cdots,v_n\}$, $U_1,\cdots,U_n$ be mutually nonempty finite sets, the graph $H$ with vertex set  $U_1\cup\cdots\cup U_n$ as follows: for each $i$, the vertices of $U_i$ are either all mutually adjacent ($U_i$ is a clique), or all mutually nonadjacent($U_i$ is an independent set); for $1\leq i,j\leq n$ and $i\neq j$, each vertex of $U_i$ is adjacent to each vertex of $U_j$ if and only if $v_i$ and $v_j$ are adjacent in $G$. We call $H$ is a mixed extension of $G$. Furthermore, we represent a mixed extension by an $n$-tuple $(t_1,...,t_n)$ of nonzero integers, where $t_i > 0$ indicates that $U_i$ is a clique of order $t_i$, and $t_i < 0$ means that $U_i$ is an independent set of order $-t_i$. For example,  $S_{k+1}$ is a star with the
 vertex $v_0$ of degree $k$ and the other vertices $v_1, v_2,..., v_k$,  $S(t_0,t_1,\cdots,t_k)$ is the mixed extension of the star
 $S_{k+1}$ of the type $(t_0,t_1,\cdots,t_k)$, where $t_i (0 \leq i \leq k)$ is an integer.
 
 \begin{lemma}{\rm(\!\!\cite{wang3})}\label{w2}
 Let $G$ be a connected graph. Then $G$ has exactly one positive $\mathcal{A}$-eigenvalue if and only if $G\cong S(t_0,-p,t_1,\cdots,t_q)$ with $p,q\geq 0, t_0\geq 1$ and $t_i\geq 2 (1\leq i \leq q)$. Furthermore,  $m(-1)=t_0-1$.
 	
 \end{lemma}
% \begin{lemma}{\rm(\!\!\cite{lei})}\label{ll}
 %	A graph $G$ possesses one negative $\mathcal{A}$-eigenvalue if and only if $G = P_2$.
% \end{lemma}
 
% \begin{lemma}\label{gl}
% 	If a real polynomial is arranged in ascending or descending powers, its number of positive roots is no more than the number of sign variations in consecutive coefficients.
% \end{lemma}
 
 \begin{lemma}{\rm(\!\!\cite{wei})}\label{ww}
 		Let $G$ be a connected graph with order $n$. Then $\xi_n(G)=-1$ if and only if $G\cong K_n$.
 \end{lemma}
 
 %\begin{lemma}{\rm(\!\!\cite{wang})}\label{kp1}
%Let $G=K_r\vee K_{n_1,n_2,\cdots,n_k}$ be a complete $(r+k)$-partite graph with order $n=r+\sum_{i=1}^{k}n_i$, $n_1\geq n_2\geq \cdots n_k \geq 2$ and $r,k\geq 1$. Then
%\begin{equation*}
%	P(\mathcal{A}(G),\lambda)=(\lambda+1)^{r-1}(\lambda+2)^{n-r-k}[(\lambda -r+1)\prod_{i=1}^{k}(\lambda-2n_i+2)-\sum_{j=1}^{k}(rn_j\prod_{i\neq j}(\lambda-2n_i+2))].
%\end{equation*} 
 %\end{lemma}
 
  \begin{lemma}{\rm(\!\!\cite{wang})}\label{kp5}
 	Let $G\cong K_r\vee K_{n_1,n_2,\cdots,n_k}$ be a complete $(r+k)$-partite graph with  $n_1\geq n_2\geq \cdots n_k \geq 2$, $r\geq 1$ and $k\geq 2$. Then
$m(-1)=r-1$.
 \end{lemma}
 \iffalse
  \begin{lemma}{\rm(\!\!\cite{wang})}\label{kp4}
	Let $G$ be a connected graph with order $n$ and diameter $d=2$. Then $\xi_{n}(G)\leq -2$, with equality if and only if $n_i\geq 2$ for $1\leq i\leq k$, 	and

	$(a)$ $G\cong K_{n_1,\cdots,n_k}$, or 
 	 	
	$(b)$ $G\cong K_r \vee (n-r)K_1$, where $1\leq r\leq n-2$, or
 	
	$(c)$ $G\cong K_r\vee K_{n_1,\cdots,n_k}$, where $k\in\begin{cases}
		\{2,3,4\},&\text{if}\ r=1; \\
		\{2,3\},&\text{if}\ r=2; \\
		\{2\},&\text{if}\ r\geq3.
	\end{cases}$

\end{lemma}
 \fi
 \begin{lemma}{\rm(\!\!\cite{wei2})}\label{kp2}
 	Let $G\cong K_r\vee K_{n_1,n_2,\cdots,n_k}$ be a complete $(r+k)$-partite graph with order $n=r+\sum\limits_{i=1}^{k}n_i$, $n_1\geq n_2\geq \cdots n_k \geq 2$. Then
 	
 	{\rm \item(i)} if  $r,k\geq 1$, then $m(-2)=n-r-k$ if and only if $(r,k)\not\in \{(1,4),(2,3)\}$. Furthermore, if $n_{s}=\cdots=n_{s+t-1} (s\geq 1,s+t-1\leq k)$, then $m(2n_{s}-2)=t-1$.
 	
 	{\rm \item(ii)} if $r=0,k\geq 2$, then $Spec_{\mathcal{A}}(G)= \left\{
 	\begin{array}{cccc} -2&2n_1-2&\cdots&2n_k-2 \\ n-k&1&\cdots&1 \end{array}
 	\right\}$.
 \end{lemma}
 
 \section{\textbf{The proof of Theorem \ref{1}}}

    \hspace{1.5em} In this section, we will prove Theorem \ref{1}.
   
   \begin{lemma}\label{mz}
   Let $A=(a_{ij})_{n\times n}$ be a real symmetric matrix, $a(\geq2)$ be an integer, $a_{ij}=\begin{cases}
   	 s_{ij}a,&\ i\neq j ;\\
   	1,&\ i=j,
   \end{cases}$ where $s_{ij}\in \{0,1\}$. Then $rank(A)=n$.
      \end{lemma}
      
      \begin{proof}
      	Firstly, we transform $A$ to a row echelon form by Gauss transformation, which takes $n-1$ steps.
      	
      	\textbf{Step 1}: Add $-a_{i1}$ times the first row to the $i$-th row of $A$ for $2\leq i\leq n$, and we obtain
      	 $A_1=(a^{(1)}_{ij})_{n\times n}$, where $a^{(1)}_{ij}=\begin{cases}
      	 	1,&\ i= j=1 ;\\
      	 	1-s^{(1)}_{ii}a,&\ 2\leq i=j\leq n;\\
      	 	0,&2\leq i\leq n, j=1;\\
      	 	s^{(1)}_{ij}a,&i\neq j, 2\leq i,j\leq n,
      	 \end{cases}$ 
      	 $s^{(1)}_{ii}=s^2_{1i}a$ for $2\leq i\leq n$, $s^{(1)}_{ij}=s_{ij}-s_{i1}s_{1j}a$ for $i\neq j$ and $ 2\leq i,j\leq n$. 
      	 
      	  Clearly, for $2\leq i\leq n$, $a^{(1)}_{ii}=1-s^{(1)}_{ii}a\neq 0$ by $s^{(1)}_{ii}$ is an integer and $a\geq 2$.
      	  
      	  \textbf{Step 2}: Add $-\frac{s^{(1)}_{i2}a}{1-s^{(1)}_{22}a}$ times the second row to the $i$-th row of $A_1$ for $3\leq i\leq n$, and we obtain $A_2=(a^{(2)}_{ij})_{n\times n}$, where $a^{(2)}_{ij}=\begin{cases}
      	  	a^{(1)}_{ij},&\ i\leq 2\ \text{or}\ j=1 ;\\
      	  	0,&\ 3\leq i\leq n,j=2;\\
      	  	\frac{1-s^{(2)}_{ii}a}{1-s^{(2)}_{22}a},&3\leq i=j\leq n;\\
      	  	\frac{s^{(2)}_{ij}a}{1-s^{(2)}_{22}a},&i\neq j, 3\leq i,j\leq n,
      	  \end{cases}$ $s^{(2)}_{22}=s^{(1)}_{22}$, $s^{(2)}_{ii}=s^{(1)}_{i2}s^{(1)}_{2i}a+s^{(1)}_{22}+s^{(1)}_{ii}-s^{(1)}_{ii}s^{(1)}_{22}a$ for $3\leq i\leq n$,  $s^{(2)}_{ij}=s^{(1)}_{ij}-s^{(1)}_{ij}s^{(1)}_{22}a-s^{(1)}_{i2}s^{(1)}_{2j}a$ for $i\neq j$ and $ 3\leq i,j\leq n $.
      	  
      	  Clearly, for $3\leq i\leq n$, $a^{(2)}_{ii}=\frac{1-s^{(2)}_{ii}a}{1-s^{(2)}_{22}a}\neq 0$ by $s^{(2)}_{ii}$ is an integer and $a\geq 2$.

      	  \textbf{Step 3}: Add $-\frac{s^{(2)}_{i3}a}{1-s^{(2)}_{33}a}$ times the third row to the $i$-th row of $A_2$ for $4\leq i\leq n$, 
      	   and we obtain $A_3=(a^{(3)}_{ij})_{n\times n}$,  where $a^{(3)}_{ij}=\begin{cases}
      	   	a^{(2)}_{ij},&\ i\leq 3\ \text{or}\ j\leq 2 ;\\
      	   	0,&\ 4\leq i\leq n,j=3;\\
      	   	\frac{1-s^{(3)}_{ii}a}{1-s^{(3)}_{33}a},&4\leq i=j\leq n;\\
      	   	\frac{s^{(3)}_{ij}a}{1-s^{(3)}_{33}a},&i\neq j, 4\leq i,j\leq n,
      	   \end{cases}$
      	  $s^{(3)}_{33}=s^{(2)}_{22}+s^{(2)}_{33}-s^{(2)}_{22}s^{(2)}_{33}a$, $s^{(3)}_{ii}=s^{(2)}_{i3}s^{(2)}_{3i}a+s^{(2)}_{33}+s^{(2)}_{ii}-s^{(2)}_{ii}s^{(2)}_{33}a$  for $4\leq i\leq n$,  $s^{(3)}_{ij}=s^{(2)}_{ij}-s^{(2)}_{ij}s^{(2)}_{33}a-s^{(2)}_{i3}s^{(2)}_{3j}a$ for $i\neq j$ and $ 4\leq i,j\leq n$.
      	  
      	  Clearly, for $4\leq i\leq n$, $a^{(3)}_{ii}=\frac{1-s^{(3)}_{ii}a}{1-s^{(3)}_{33}a}\neq 0$ by $s^{(3)}_{ii}$ is an integer and $a\geq 2$.
      	  
      	  Continuing the above steps, we get row echelon form $A_{n-1}=(a^{(n-1)}_{ij})_{n\times n}$ of $A$. Clearly,  we have $a^{(n-1)}_{ii}=\begin{cases}
      	  	1,&\ i=1 ;\\
      	  	a^{(i-1)}_{ii},&\ 2\leq i\leq n,
      	  \end{cases}$ which implies $a^{(n-1)}_{ii}\neq 0$ for $1\leq i\leq n$. Thus $rank(A)=rank(A_{n-1})=n$.
      \end{proof}
     %\[
     %\begin{pmatrix}
     %	a_{11} & a_{12} & \cdots & a_{1n} \\
     %	a_{21} & a_{22} & \cdots & a_{2n} \\
     %	\vdots & \vdots & \ddots & \vdots \\
     %	a_{m1} & a_{m2} & \cdots & a_{mn}
     %\end{pmatrix}
     %\tag*{\centering ($1$)}
     %\] 
\begin{lemma}\label{qm}
		Let $M$ be an $ n\times n$ real symmetry matrix, $M^*$ be its $k\times k$  principal submatrix,   %$\xi$ be an eigenvalue of $M^*$ with 
		$m_{M^*}(\xi)=s$ with $s\geq 0$. Then $m_{M}(\xi)\leq n-k+s$.
\end{lemma}      
    \begin{proof}
    	We will complete the proof by the following three cases.
    	
    	\textbf{Case 1}: $m_{M}(\xi)=0$.
    	
    	Clearly, $m_{M}(\xi)\leq n-k+s$ by $0\leq s\leq k\leq n$.
    	
    	\textbf{Case 2}: $m_{M}(\xi)>0$ and $s\geq 1$.

   Without loss of generality, we take $ \xi_{t+1}(M^*)=\cdots=\xi_{t+s}(M^*)=\xi$ for $0\leq t\leq k-s$. 
   
   If $1\leq t\leq k-s-1$, then by Lemma \ref{pri}, we have $\xi_1(M)\geq\xi_2(M)\geq\cdots\geq\xi_t(M)\geq\xi_t(M^*)>\xi>\xi_{t+s+1}(M^*)\geq \xi_{n-k+t+s+1}(M)\geq \cdots \geq \xi_{n-1}(M) \geq \xi_n(M)$, which implies $m_{M}(\xi)\leq n-t-[n-(n-k+t+s+1)+1]=n-k+s$. 
   
   If $t=0$ or $t=k-s$, similar to the above proof, we have $m_{M}(\xi)\leq n-k+s$.
   
   \textbf{Case 3}: $m_{M}(\xi)>0$ and $s=0$.
    	
    	 By $s=0$, we have $\xi>\xi_{1}(M^*)$ or $ \xi_{1}(M^*)>\xi> \xi_{k}(M^*)$  or $\xi<\xi_{k}(M^*)$. 
    	 
    	 If $ \xi_{1}(M^*)>\xi> \xi_{k}(M^*)$, then there exists $1\leq l\leq k-1$ satisfying $ \xi_{l}(M^*)>\xi> \xi_{l+1}(M^*)$. By Lemma \ref{pri}, we have $\xi_1(M)\geq\xi_2(M)\geq\cdots\geq\xi_l(M)\geq\xi_l(M^*)>\xi>\xi_{l+1}(M^*)\geq \xi_{n-k+l+1}(M)\geq \cdots \geq \xi_{n-1}(M)\geq \xi_n(M)$, which implies $m_{M}(\xi)\leq n-l-[n-(n-k+l+1)+1]=n-k$. 
    	 
    	 If $\xi>\xi_{1}(M^*)$ or $\xi<\xi_{k}(M^*)$, similar to the above proof, we have  $m_{M}(\xi)\leq n-k$.
    \end{proof} 
    \vskip0.15cm
    
    The following Proposition is very  useful, but is so easy that we ignore its proof.

      \begin{proposition}\label{p1}
      	Let $G$ be a connected graph with $diam(G)=d$, $V_i=\{v\in V(G)|\varepsilon(v)=i\}$. Then  
      	
      	{\rm \item(i)} $V(G)=V_1\cup V_2\cup \cdots \cup V_{d}$ is a partition of the vertex set of $G$;
      	
      	 {\rm \item(ii)} if $V_1\neq \emptyset $, then $1\leq d\leq2$;
      	 
      	 {\rm \item(iii)}  $d=1$ if and only if $G\cong K_n$.
      \end{proposition}

      \begin{lemma}\label{qb1}
      	Let $G (\not\cong K_n)$ be a connected graph with order $n$, $V_1\neq \emptyset$, $m(-1)=n-k$.  Then $|V_1|=n-k+1$ or $n-k$.
      \end{lemma}
      \begin{proof}
      	Clearly, $|V_1|\leq n-k+1$ by $V_1\neq \emptyset$,  $m(-1)=n-k$ and Lemma \ref{gx}. Now we show  $|V_1|\geq n-k$.
      	
      	By $G \not\cong K_n$, we have $d\geq 2$, and $d\leq 2$ by $V_1\neq \emptyset$. Then $d=2$ and $V(G)=V_1\cup V_2$ is a partition of $V(G)$.

      	  If $|V_1|\leq n-k-1$, then  $|V_2|\geq k+1$ by $|V_1|+|V_2|=n$. Let $B$ be the principal submatrix of $\mathcal{A}(G)+I$ indexed by $V_2$. Then $rank(B)=|V_2|$  by Lemma \ref{mz}, and thus $m_{B-I}(-1)=m_{B}(0)=|V_2|-rank(B)=0$ by matrix theory. Therefore, $m(-1)\leq n-|V_2|+m_{B-I}(-1f)=n-|V_2|\leq n-k-1$ by Lemma \ref{qm}, which implies a contradiction.
      \end{proof}
      %$diam(G)=d\geq 2$. Then there no graph $G$ with  $m(-1)=k$ when order $n\geq kd+2$.
      
      \begin{lemma}\label{qb2}
      	Let $G$ be a connected graph with order $n$ and $diam(G)=d$. If $V_1=\emptyset$, and $n\geq k(d-1)+1$, then $m(-1)\leq n-k-1$.
      \end{lemma}
      \begin{proof}
      	 By $V_1=\emptyset$ and $diam(G)=d$, we have $V(G)=V_2\cup V_3\cup\cdots\cup V_d$ is a partition of the vertex set of $G$. Clearly, there exists $|V_i|\geq k+1$ by $n\geq k(d-1)+1$. 
      	
      	Let $B$ be the principal submatrix of $\mathcal{A}(G)+I$ indexed by $V_i$. Then $rank(B)=|V_i|$ by Lemma \ref{mz}, and thus $m_{B-I}(-1)=m_{B}(0)=|V_i|-rank(B)= 0$  by matrix theory. Therefore, $m(-1)\leq n-|V_i|+m_{B-I}(-1)=n-|V_i|\leq n-k-1$ by Lemma \ref{qm}.
      \end{proof}
      
      \begin{theorem}\label{n-1}
      	Let $G$ be a connected graph with order $n$. Then $m(-1)=n-1$ if and only if $G\cong K_n$. In addition, $G$ is determined by its $\mathcal{A}$-spectrum.
      \end{theorem}
      \begin{proof}
      	Clearly,  we have $m_{\mathcal{A}(K_n)}(-1)=n-1$. 
      	On the other hand, if $m(-1)=n-1$, then  $\xi_2=\cdots=\xi_{n}=-1$ and $\xi_1=n-1$ by $tr(\mathcal{A}(G))=0$, and thus $G\cong K_n$ by Lemma \ref{ww}. Therefore, $m(-1)=n-1$ if and only if $G\cong K_n$.
      	
      	 It is obvious that $G$ is determined by its $\mathcal{A}$-spectrum.
      \end{proof}

       \begin{theorem}\label{n-2}
      	There is no connected graph $G$ with order $n$ and  $m(-1)=n-2$.
      \end{theorem}
\begin{proof}
	Suppose  there exists a connected graph $G$ with order $n$ satisfying $m(-1)=n-2$, then $\mathcal{A}(G)$ exactly has a positive $\mathcal{A}$-eigenvalue or two positive $\mathcal{A}$-eigenvalues.
	
	If $\mathcal{A}(G)$ exactly has a positive $\mathcal{A}$-eigenvalue, then $G\cong S(t_0,-p,t_1,\cdots,t_q)$ with $p,q\geq 0$, $t_0\geq 1$, $t_i\geq 2$ for $1\leq i \leq q$ and $m(-1)=t_0-1$ by Lemma \ref{w2}. Thus we have $t_0=n-1$, and $G\cong K_n$, which implies a contradiction by Theorem \ref{n-1}.
	
	If $\mathcal{A}(G)$ exactly has two positive $\mathcal{A}$-eigenvalues, then $\xi_{n}=-1$, and thus  $G\cong K_n$ by Lemma \ref{ww}, which implies $m(-1)=n-1$, a contradiction.
\end{proof}

\iffalse
 \begin{lemma}\label{n-31}
	Let $G$ be a connected graph with order $n$ and $diam(G)=d$, $m(-1)=n-3$. Then $2\leq d\leq 3$.
\end{lemma}
\begin{proof}
	Clearly, we have $d\geq 2$ by $d=1$ if and only if $G\cong K_n$ and $m_{\mathcal{A}(K_n)}(-1)=n-1$. Now we show $d\leq 3$.
	
	 Let $d(v_0,v_d)=d$ and $P_{d+1}=v_0v_1\cdots v_{d-1}v_d$ be a path with length $d$. Then $\varepsilon(v_0)=\varepsilon(v_d)=d$, $d-1\leq\varepsilon(v_1),\varepsilon(v_{d-1})\leq d$. Clearly, the principal submatrix of $\mathcal{A}(G)+I$ indexed by $\{v_0,v_1,v_{d-1},v_d\}$ is
	\begin{equation*}			
		B_1=\begin{pmatrix}
			1&0&a&d\\
			0&1&0&b\\
			a&0&1&0\\
			d&b&0&1		
		\end{pmatrix},
	\end{equation*}
	where $a,b\in \{0,d-1\}$.
	By direct calculation, we have 
	%$ P(B_1,\lambda)=\lambda^4-4\lambda^3-(a^2+b^2+d^2-6)\lambda^2+(2a^2+2b^2+2d^2-4)\lambda+a^2b^2-a^2-b^2-d^2+1$. Thus 
	$ P(B_1,0)=a^2b^2-a^2-b^2-d^2+1$.
	
	 If $d\geq 4$, we have $ P(B_1,0)\neq 0$, which implies $m_{B_1-I}(-1)=0$. Therefore,  $m(-1)\leq n-4$ by Lemma \ref{qm}, a contradiction.
\end{proof}\fi 
\begin{lemma}\label{n-41}
	Let $G$ be a connected graph with order $n$ and $ diam(G)=d\geq 4$. Then $m(-1)\leq n-5$.
\end{lemma}

\begin{proof}
 Let $d(v_0,v_d)=d$ and $P_{d+1}=v_0v_1v_2\cdots v_d$ be a path with length $d$. Then $\varepsilon(v_0)=\varepsilon(v_d)=d$, $d-1\leq\varepsilon(v_1),\varepsilon(v_{d-1})\leq d$, $d-2\leq\varepsilon(v_{2})\leq d$. 	Clearly,  the principal submatrix of $\mathcal{A}(G)+I$ indexed by $\{v_0,v_1,v_2,v_{d-1},v_d\}$ is
	\begin{equation*}			
		B_1=\begin{pmatrix}
			
			1&0&a&b&d\\
			0&1&0&0&c\\
			a&0&1&0&e\\
			b&0&0&1&0\\
			d&c&e&0&1		
		\end{pmatrix},
	\end{equation*}
	where $a\in \{0,2\},b,c\in\{0,d-1\},e\in\{0,d-2\}$. By direct calculation, we have $P(B_1,0)=-a^2c^2-b^2c^2-b^2e^2-2ade+a^2+b^2+c^2+d^2+e^2-1$. %Now we complete the proof by the following four cases.
	
	\textbf{Case 1}: $\varepsilon(v_2)=d-2, \varepsilon(v_1)=\varepsilon(v_{d-1})=d-1$.
	
	 Then $b=c=d-1,e=d-2$, thus $P(B_1,0)=-a^2d^2-2d^4+2a^2d-2ad^2+10d^3+4ad-15d^2+8d<0$ by $a\in \{0,2\}$ and $ d\geq 4$.
	
	\textbf{Case 2}: $\varepsilon(v_2)=d-2, \varepsilon(v_1)=d-1,\varepsilon(v_{d-1})=d$.
	
	 Then $b=0,c=d-1,e=d-2$, thus $P(B_1,0)=-a^2d^2+2a^2d-2ad^2+4ad+3d^2-6d+4\neq 0$ by $a\in \{0,2\}$ and $ d\geq 4$.
	
	\textbf{Case 3}: $\varepsilon(v_2)=d-2, \varepsilon(v_1)=d,\varepsilon(v_{d-1})=d-1$.
	
	 Then $b=d-1,c=0,e=d-2$, thus $P(B_1,0)=-d^4-2ad^2+6d^3+a^2+4ad-10d^2+6d\neq 0$ by $a\in \{0,2\}$ and $ d\geq 4$.
	
	\textbf{Case 4}: $\varepsilon(v_2)=d-2, \varepsilon(v_1)=\varepsilon(v_{d-1})=d$. 
	
	Then $b=c=0,e=d-2$, thus $P(B_1,0)=-2ad^2+4ad+a^2+2d^2-4d+3\neq 0$  by $a\in \{0,2\}$ and $ d\geq 4$.
	
	\textbf{Case 5}: $\varepsilon(v_2)\geq d-1$.
	
	 Then $a=e=0$, thus $P(B_1,0)=-b^2c^2+b^2+c^2+d^2-1\neq 0$  by $b,c\in \{0,d-1\}$ and $ d\geq 4$. 
	
	Combining the above arguments, we have $P(B_1,0)\neq 0$, which implies $m_{B_1-I}(-1)=0$. Therefore, $m(-1)\leq n-5$ by Lemma \ref{qm}.
\end{proof}
\vskip0.15cm
By Proposition \ref{p1}, Theorem \ref{n-1} and Lemma \ref{n-41}, we have the following conclusion immediately.
\begin{remark}\label{r1}
	Let $G$ be a connected graph with order $n$, $diam(G)=d$ and $n-3\leq m(-1)\leq n-4$. Then $2\leq d\leq 3$.
\end{remark}

\begin{lemma}\label{r-1}
	Let $G\cong K_r\vee (n-r)K_1$ be a complete $(r+1)$-partite graph with $1\leq r\leq n-2$. Then $m(-1)=r-1$.
\end{lemma}
\begin{proof}
	Clearly, the eccentricity matrix $\mathcal{A}(G)$ can be written as
	\begin{equation*}			
		\begin{pmatrix}
			
			J_{r\times r}-I_{r\times r}&J_{r\times (n-r)}\\
			J_{(n-r)\times r}&2J_{(n-r)\times (n-r)}-2I_{(n-r)\times (n-r)}\\		
		\end{pmatrix},
	\end{equation*}
	 and the equitable quotient matrix of $\mathcal{A}(G)$ is
	 	\begin{equation*}			
	 	Q=\begin{pmatrix}
	 		
	 		r-1&n-r\\
	 		r&2(n-r-1)\\		
	 	\end{pmatrix}.
	 \end{equation*}
	 Then $P(\mathcal{A}(G),\lambda)=(\lambda+1)^{r-1}(r+2)^{n-r-1}P(Q,\lambda)$ by Lemma \ref{qjz}. By direct calculation, we have $P(Q,-1)=r(n-r-1)\neq 0$, which implies $m(-1)=r-1$. 
\end{proof}

 \begin{theorem}\label{n-32}
	Let $G$ be a connected graph with order $n\geq 4$. Then $m(-1)=n-3$ if and only if  $G\in \{P_4, K_{n-2}\vee2K_1\}$. In addition, $G$ is determined by its $\mathcal{A}$-spectrum.
\end{theorem}
\begin{proof}
		By direct calculation, we have $P(\mathcal{A}(P_4),\lambda)=(\lambda+1)(\lambda-1)(\lambda+4)(\lambda-4)$,  then $m_{\mathcal{A}(P_4)}(-1)=1$.  If  $G\cong K_{n-2}\vee2K_1$, then  $m_{\mathcal{A}(G)}(-1)=n-3$  by Lemma \ref{r-1} and $r=n-2$.
	
	Let $m(-1)=n-3$. Now we show $G\in \{P_4, K_{n-2}\vee2K_1\}$.
	
	Clearly, we have $2\leq d\leq 3$ by Remark \ref{r1}. Then we complete the proof by the following two cases.
	
	\textbf{Case 1}: $d=3$.

	  Let $d(v_0,v_3)=3$ and $P_{4}=v_0v_1v_2v_3$ be a path with length $3$. Then $\varepsilon(v_0)=\varepsilon(v_3)=3$, $2\leq\varepsilon(v_1),\varepsilon(v_2)\leq 3$. Firstly, we show $\varepsilon(v_1)=\varepsilon(v_2)=2$.
	 
	 Let $\varepsilon(v_1)=3$, and $B_2$ be  the principal submatrix of $\mathcal{A}(G)+I$ indexed by $\{v_0,v_1,v_{2},v_3\}$, where $x\in \{0,2\}$ and
	 \begin{equation*}			
	 	B_2=\begin{pmatrix}
	 		1&0&x&3\\
	 		0&1&0&0\\
	 		x&0&1&0\\
	 		3&0&0&1		
	 	\end{pmatrix}.
	 \end{equation*}
	  
	  By direct calculation, we have $ P(B_2,0)=-x^2-8<0$, which implies $m_{B_2-I}(-1)=0$. Thus $m(-1)\leq n-4$ by Lemma \ref{qm}, a contradiction. Therefore, $\varepsilon(v_1)=2$. 
	  
	  Similar to the above proof, we have $\varepsilon(v_2)=2$.  
	  
	  Now we complete the proof of Case 1 by the following two Subcases.
	 
	 \textbf{Subcase 1.1}: $n=4$.
	 
	 It is obvious that  $G\cong P_4$.
	 
	 \textbf{Subcase 1.2}: $n\geq 5$.
	 
	 Then there exists a vertex, say, $v_4\in V(G)\setminus \{v_0,v_1,v_2,v_3\}$, satisfying $\varepsilon(v_4)\in\{2,3\}$ and there exists some $v_i(0\leq i\leq3)$ such that $v_i\sim v_4$.  Then the induced subgraph $G[\{v_0,v_1,v_2,v_3,v_4\}]$ must be one of $H_2,H_3,\cdots,H_7$, as shown in Figure $2$.
	 
	  \begin{figure}[!h]
	 	\centering
	 	\begin{tikzpicture}
	 		\node[anchor=south west,inner sep=0] (image) at (0,0) {\includegraphics[width=1\textwidth]{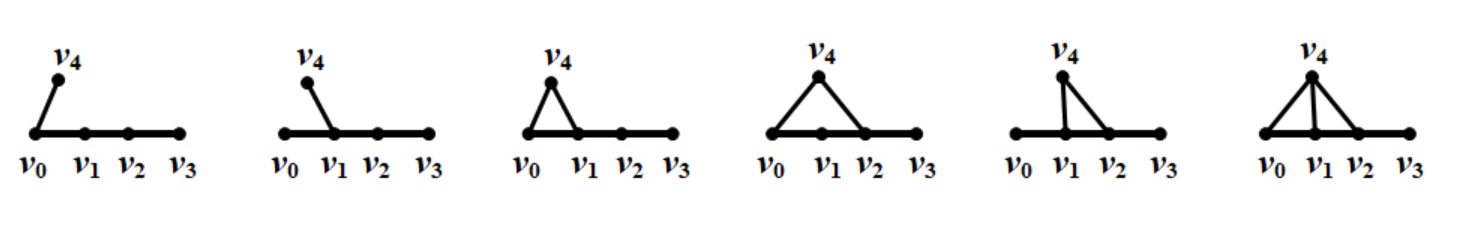}};
	 		\begin{scope}[
	 			x={(image.south east)},
	 			y={(image.north west)}
	 			]
	 			\node [black, font=\bfseries] at (0.08,-0.1) {$H_2$ };
	 			\node [black, font=\bfseries] at (0.25,-0.1) {$H_3$};
	 			\node [black, font=\bfseries] at (0.42,-0.1) {$H_4$};
	 			\node [black, font=\bfseries] at (0.58,-0.1) {$H_5$};
	 			\node [black, font=\bfseries] at (0.75,-0.1) {$H_6$};
	 			\node [black, font=\bfseries] at (0.92,-0.1) {$H_7$};
	 		\end{scope}
	 	\end{tikzpicture}
	 	
	 	\caption{Graphs $H_2,H_3,H_4,H_5,H_6,H_7$.}\label{Figure 2}
	 \end{figure}
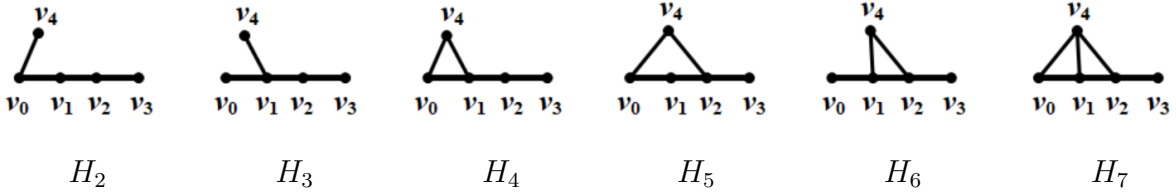

	 Let $B_i$ $(3\leq i\leq 8)$ be  the principal submatrix of $\mathcal{A}(G)+I$ indexed by $\{v_0,v_1,v_2,v_3,v_4\}$ of $H_{i-1}$. Then
	  \[			
	 B_3=\begin{pmatrix}
	 	1&0&2&3&0\\
	 	0&1&0&2&2\\
	 	2&0&1&0&2\\
	 	3&2&0&1&a\\
	 	0&2&2&a&1\\
	 	
	 \end{pmatrix}
	 \!
	 ,B_4=\begin{pmatrix}
	 	1&0&2&3&b\\
	 	0&1&0&2&0\\
	 	2&0&1&0&2\\
	 	3&2&0&1&a\\
	 	b&0&2&a&1\\
	 	
	 \end{pmatrix}
	 \!
	 ,B_5=\begin{pmatrix}
	 	1&0&2&3&0\\
	 	0&1&0&2&0\\
	 	2&0&1&0&2\\
	 	3&2&0&1&a\\
	 	0&0&2&a&1\\
	 	
	 \end{pmatrix},
	 \]
	 
	 \[			
	 B_6=\begin{pmatrix}
	 	1&0&2&3&0\\
	 	0&1&0&2&2\\
	 	2&0&1&0&0\\
	 	3&2&0&1&b\\
	 	0&2&0&b&1\\
	 	
	 \end{pmatrix}
	 \!
	 ,B_7=\begin{pmatrix}
	 	1&0&2&3&b\\
	 	0&1&0&2&0\\
	 	2&0&1&0&0\\
	 	3&2&0&1&b\\
	 	b&0&0&b&1\\
	 	
	 \end{pmatrix}
	 \!
	 ,B_8=\begin{pmatrix}
	 	1&0&2&3&0\\
	 	0&1&0&2&0\\
	 	2&0&1&0&0\\
	 	3&2&0&1&b\\
	 	0&0&0&b&1\\
	 	
	 \end{pmatrix},
	 \]
	  where $a\in \{0,2,3\}$, $b\in \{0,2\}$.
	 By direct calculation, for $3\leq i\leq 8$, we have the coefficient of $\lambda$ in $P(B_i,\lambda)$ is not equal to $0$, then $m_{B_i}(0)\leq 1$, which implies$m_{B_i-I}(-1)\leq 1$ and thus $m(-1)\leq n-4$ by Lemma \ref{qm}, a contradiction. For example, if $i=3$, then $P(B_3,\lambda)=\lambda^5-5\lambda^4-15\lambda^3+(65-4a)\lambda^2+(50-4a)\lambda+24a-192$. Clearly, we have $50-4a\neq 0$ by $a\in \{0,2,3\}$.

\textbf{Case 2}: $d=2$.

 %Let $H$ be an induced subgraph of $G$. For any vertices $u, v \in V(H)$, we have  $\varepsilon_{H}(v)=\varepsilon_{G}(v)$ and $d_H (u, v) = d_G(u, v)$ by $diam(G)=2$. Thus $\mathcal{A}(H)$ is a principal submatrix of $\mathcal{A}(G)$ by Lemma \ref{w}. The following Claim 2 is true.

Firstly, we show $G$ is $\overline{P_3}$-free. Otherwise, there exist three vertices $v_1,v_2,v_3$, such that $G[\{v_1,v_2,v_3\}]=\overline{P_3}$. Without loss of generality, we suppose $v_1\sim v_2,v_1\not\sim v_3$ and $v_2\not\sim v_3$. 

Clearly, by $diam(G)=2$, there exists $v_4\in V(G)$, such that $v_4\sim v_2$ and $v_4\sim v_3$. Let $B_9$ be  the principal submatrix of $\mathcal{A}(G)+I$ indexed by $\{v_1,v_2,v_3,v_4\}$. Then
\begin{equation*}			
	B_9=\begin{pmatrix}
		
		1&0&2&c\\
		0&1&2&e\\
		2&2&1&e\\
		c&e&e&1\\		
	\end{pmatrix},
\end{equation*}
where $c\in \{0,1,2\},e\in \{0,1\}$. 

By direct calculation, we have $P(B_9,\lambda)=\lambda^4-4\lambda^3-(c^2+2e^2+2)\lambda^2+(2c^2-4ce+12)\lambda+3c^2-4ce+6e^2-7$, and $P(B_9,0)=3c^2-4ce+6e^2-7\neq 0$ by  $c\in \{0,1,2\},e\in \{0,1\}$. Then $m_{B_9}(0)= 0$, which implies $m_{B_9-I}(-1)= 0$ and thus $m(-1)\leq n-4$ by Lemma \ref{qm}, a contradiction. Therefore, $G$ is $\overline{P_3}$-free.

By Lemma \ref{fr} and $G$ is $\overline{P_3}$-free, $G$ is a  complete multipartite graph, say, $G\cong K_{n_1,n_2,\cdots,n_k}$, where $n_1\geq n_2\geq \cdots\geq n_k\geq 1$ and $k \geq 2$.

\textbf{Subcase 2.1}: $n_k\geq2$.

Then $m(-1)=0$ by (ii) of Lemma \ref{kp2}, a contradiction.

\textbf{Subcase 2.2}: $n_k=1$.

Without loss of generality, we assume that $n_{k-r+1}=\cdots=n_k=1$, and $n_1\geq n_2\geq\cdots\geq n_{k-r}\geq 2$ if $k-r>0$.

If $k-r\geq2$, then $m(-1)=r-1=(r-\sum\limits_{i=1}^{k-r}n_i)-1\leq(n-n_1-n_{k-r})-1\leq n-5$ by  Lemma \ref{kp5}, a contradiction.

If $k-r=1$, then $G= K_{n-r,1,\cdots,1}\cong K_r\vee (n-r)K_1$ with $1\leq r\leq n-2$ and $m(-1)=r-1$ by Lemma \ref{r-1}, and thus $r=n-2$ by $m(-1)=n-3$. Therefore, $G\cong K_{n-2}\vee2K_1$.

If $k-r=0$, then $G=K_{1,1,\cdots,1}\cong K_n$, and thus $m(-1)=n-1$ by Theorem \ref{n-1}, a contradiction.

Combining the above two subcases, we have $G\cong K_{n-2}\vee2K_1$ if $d=2$, and $m(-1)=n-3$ if and only if $G\in \{P_4, K_{n-2}\vee2K_1\}$ by the above two cases.

Clearly, $G$ is determined by its $\mathcal{A}$-spectrum.
\end{proof}

	\begin{table}[h]\label{11}
	\centering
	
	\renewcommand{\arraystretch}{1.5}
	
	\begin{tabular}{|c|c|}
		\hline
		graph $G$ & $P(\mathcal{A}(G),\lambda)$ \\ \hline
		$K_{n-4}\vee 4K_1$&$(\lambda+1)^{n-5}(\lambda+2)^{3}[\lambda^2-(1+n)\lambda+2n-14]$\\ \hline
		$K_{n-4}\vee(2K_1\cup K_2)$ & $(\lambda+1)^{n-5}(\lambda+2)\lambda[\lambda^3-(n-3)\lambda^2-(2n+10)\lambda+4n-32]$  \\ \hline
		$K_{n-4}\vee( P_3\cup K_1)$ & $(\lambda+1)^{n-5}(\lambda+2)[\lambda^4-(n-3)\lambda^3-(2n+6)\lambda^2+(4n-20)\lambda+8]$ \\ \hline
		$K_{n-4}\vee2K_2$&$(\lambda+1)^{n-5}\lambda^2(\lambda+2)[\lambda^2-(n-3)\lambda-2n+6]$\\ \hline
		$K_{n-4}\vee \renewcommand{\arraystretch}{2} P_4$&\makecell{$(\lambda+1)^{n-5}[\lambda^5-(n-5)\lambda^4-(4n-4)\lambda^3-12\lambda^2$\\$+(8n-16)\lambda+16]$}\\ \hline
		$K_{n-4}\vee(K_3\cup K_1)$&$(\lambda+1)^{n-5}\lambda^2[\lambda^3-(n-5)\lambda^2-(4n-4)\lambda-12]$\\ \hline
		$K_{n-4}\vee C_4$&$(\lambda+1)^{n-5}(\lambda+2)^2[\lambda^3+(1-n)\lambda^2+4n-12]$\\ 
		\hline
	\end{tabular}
	\caption{the characteristic polynomial of $\mathcal{A}(G)$, where $G\in \mathcal{G}_1$}.
\end{table}
\begin{lemma}\label{7}
	Let $G$ be a connected graph with order $n(\geq5)$ and $|V_1|=n-4$, $\mathcal{G}_1=\{K_{n-4}\vee4K_1,K_{n-4}\vee(2K_1\cup K_2),K_{n-4}\vee(P_3\cup K_1),K_{n-4}\vee2K_2,K_{n-4}\vee P_4,K_{n-4}\vee(K_3\cup K_1),K_{n-4}\vee C_4\}$. Then $G\in \mathcal{G}_1$ and $m(-1)=n-5$.
\end{lemma}
\begin{proof}
	Clearly, we have $d=2$ by $|V_1|=n-4$ and Proposition \ref{p1}. Thus  $V(G)=V_1\cup V_2$,  $G[V_1]=K_{n-4}$, the maximal degree of $G[V_2]$, $\bigtriangleup (G[V_2])\leq 2$, and the number of edges of $G[V_2]$, $0\leq |E(G[V_2])|\leq 4$. Therefore, $G\in\mathcal{G}_1$.  By direct calculation and Lemma \ref{qjz}, we have $P(\mathcal{A}(G),\lambda)$ as shown in  Table 1, and thus $m(-1)=n-5$.
\end{proof}

\begin{theorem}\label{n-42}
	Let $G$ be a connected graph with order $n\geq 9$. Then $m(-1)=n-4$ if and only if $G\in\{ K_{n-3}\vee (K_2\cup K_1) ,  K_{n-3}\vee 3K_1\} $. In addition, $G$ is determined by its $\mathcal{A}$-spectrum.
\end{theorem}
\begin{proof}
	By direct calculation and Lemma \ref{qjz}, we have $$P(\mathcal{A}(K_{n-3}\vee (K_2\cup K_1)),\lambda)=(\lambda+1)^{n-4}\lambda[\lambda^3-(n-4)\lambda^2-(3n-1)\lambda-8],$$ $$P(\mathcal{A}(K_{n-3}\vee 3K_1),\lambda)=(\lambda+1)^{n-4}(\lambda+2)^2(\lambda^2-n\lambda+n-7),$$ which implies $m(-1)=n-4$ if $G\in\{ K_{n-3}\vee (K_2\cup K_1) ,  K_{n-3}\vee 3K_1\} $.
	
	 Let $m(-1)=n-4$. Then $ 2\leq diam(G)=d\leq3$ by Remark \ref{r1}. If $|V_1|= \emptyset$, then $m(-1)\leq\begin{cases}
	 	n-5,&\ \text{if}\ d=2\\
	 	n-9,&\ \text{if}\ d=3
	 \end{cases}$ by $n\geq 9$ and  Lemma \ref{qb2}, a contradiction with $m(-1)=n-4$. Therefore, $|V_1|\neq \emptyset$, and thus $d=2,V(G)=V_1\cup V_2$ by Proposition \ref{p1}. Moreover, we have $n-4\leq|V_1|\leq n-3$ by Lemma \ref{qb1}.

If $|V_1|=n-4$, then $G\in \mathcal{G}_1$ and $m(-1)=n-5$ by Lemma \ref{7}, a contradiction.%then $G[V_1]=K_{n-4}$, $\bigtriangleup (G[V_2])\leq 2$ and the number of edges of $G[V_2]$ is greater than or equal to $0$ and less than or equal to $4$, and thus $G\in\mathcal{G}_1$.
 %By direct calculation and Lemma \ref{qjz}, we have $P(\mathcal{A}(G),\lambda)$ in  Table 1. Thus $m(-1)=n-5$ if  $G\in \mathcal{G}_1$, a contradiction.

% $P(\mathcal{A}(K_{n-4}\vee(2K_1\cup K_2)),\lambda)=(\lambda+1)^{n-5}(\lambda-2)\lambda[\lambda^3-(n-3)\lambda^2-(2n+10)\lambda+4n-32], P(\mathcal{A}(K_{n-4}\vee(K_1\cup (2K_1\vee K_1))),\lambda)=(\lambda+1)^{n-5}\lambda[\lambda^4-(n-3)\lambda^3-(2n+6)\lambda^2+(4n-20)\lambda+8],P(\mathcal{A}(K_{n-4}\vee(K_2\cup K_2)),\lambda)=(\lambda+1)^{n-5}\lambda^2(\lambda+2)[\lambda^2-(n-3)\lambda-2n+6], P(\mathcal{A}(K_{n-4}\vee(K_1\cup K_3)),\lambda)=(\lambda+1)^{n-5}\lambda^2[\lambda^3-(n-5)\lambda^2-(4n-4)\lambda-12], P(\mathcal{A}(K_{n-4}\vee P_4),\lambda)=(\lambda+1)^{n-5}[\lambda^5-(n-5)\lambda^4-(4n-4)\lambda^3-12\lambda^2+(8n-16)\lambda+16]$.
 
If $|V_1|=n-3$, it is obvious that $G\cong  K_{n-3}\vee (K_2\cup K_1)$ or $G\cong K_{n-3}\vee 3K_1$.   

Clearly, $G$ is determined by its $\mathcal{A}$-spectrum.
\end{proof}

\begin{lemma}\label{n-51}
	Let $G$ be a connected graph with order $n$ and $diam(G)=d$, $m(-1)=n-5$. Then $2\leq d\leq 4$.
\end{lemma}

\begin{proof}
	Clearly, we have $d= 1$ if and only if $G\cong K_n$, then $m(-1)=n-1>n-5$ by Theorem \ref{n-1}, a contradiction. So $d\geq 2$. Now we show $d\leq 4$.
	
	 Let $d(v_0,v_d)=d\geq 5$ and $P_{d+1}=v_0v_1v_2\cdots v_{d-2}v_{d-1}v_d$ be a path with length $d$. Then $\varepsilon(v_0)=\varepsilon(v_d)=d$, $d-1\leq\varepsilon(v_1),\varepsilon(v_{d-1})\leq d$, $d-2\leq\varepsilon(v_{2}), \varepsilon(v_{d-2})\leq d$. 	Clearly,  the principal submatrix of $\mathcal{A}(G)+I$ indexed by $\{v_0,v_1,v_2,v_{d-2},v_{d-1},v_d\}$ is
	\begin{equation*}			
		B_{10}=\begin{pmatrix}
			
			1&0&0&a&b&d\\
			0&1&0&0&0&c\\
			0&0&1&0&0&e\\
			a&0&0&1&0&0\\
			b&0&0&0&1&0\\
			d&c&e&0&0&1	
		\end{pmatrix},
	\end{equation*}
	where $a,e\in \{0,d-2\},b,c\in\{0,d-1\}$. By direct calculation and $d\geq5$, we have $$P(B_{10},0)=(a^2+b^2)(c^2+e^2)-a^2-b^2-c^2-d^2-e^2+1\neq 0,$$  which implies $m_{B_{10}-I}(-1)=m_{B_{10}}(0)=0$, Thus $m(-1)\leq n-6$ by Lemma \ref{qm}, a contradiction.	Therefore, $ d\leq 4$.
\end{proof}

\begin{lemma}\label{n-52}
	Let $G$ be a connected graph with order $n\geq 16$, $m(-1)=n-5$, $\mathcal{G}_1$ as defined in Lemma \ref{7}. Then $G\in \mathcal{G}_1\cup\{K_{n-5}\vee C_5,K_{n-5}\vee(K_1\cup P_4),K_{n-5}\vee H_1\}$.
\end{lemma}
\begin{proof}
	 Let $m(-1)=n-5$. Then $ 2\leq diam(G)=d\leq4$ by Lemma \ref{n-51}. 
	 
	 If $|V_1|= \emptyset$, then $m(-1)\leq\begin{cases}
		n-16,&\ d=2\\
		n-8,&\ d=3\\
		n-6,&\ d=4
	\end{cases}$ by $n\geq 16$ and  Lemma \ref{qb2}, a contradiction with $m(-1)=n-5$. Therefore, $|V_1|\neq \emptyset$, and thus $d=2,V(G)=V_1\cup V_2$ by Proposition \ref{p1}. Moreover, we have $n-5\leq|V_1|\leq n-4$ by Lemma \ref{qb1}.
	%Clearly, we have $ diam(G)=d=2$ by $n\geq 16$, Lemma \ref{n-51} and  Lemma \ref{qb2}. Thus $|V_1|\neq \emptyset$ by $n\geq 16$, $d=2$ and Lemma \ref{qb2}.
%	Furthermore, we have  $n-4\leq|V_1|\leq n-5$  by Lemma \ref{qb1}.
	
	\textbf{Case 1}: $|V_1|=n-4$.
	
	 It is obvious that $G\in \mathcal{G}_1$ by Lemma \ref{7}. 
	
	\textbf{Case 2}: $|V_1|=n-5$.
	
	 Then $|V_2|=5$ by $V(G)=V_1\cup V_2$, and thus $G\cong K_{n-5}\vee G[V_2]$, the maximal degree of $G[V_2]$, $\bigtriangleup (G[V_2])\leq 3$. Without loss of generality, we assume that $v_1\in V_1$ and $V_2=\{v_2,v_3, v_4, v_5, v_6\}$. Now we show $G\in \{K_{n-5}\vee C_5,K_{n-5}\vee(K_1\cup P_4),K_{n-5}\vee H_1\}$ by the following two subcases.
	
	\textbf{Subcase 2.1}: $G[V_2]$ is a connected graph.
	
	\textbf{Subcase 2.1.1}: $K_{1,3}$ is an induced subgraph of $G[V_2]$.
	
	Let $G[\{v_2,v_3, v_4, v_5\}]=K_{1,3}$, where $v_2\sim v_3$, $v_2\sim v_4$, $v_2\sim v_5$. Then the principal submatrix of $\mathcal{A}(G)+I$ indexed by $\{v_1,v_2,v_3,v_4,v_5,v_6\}$ is
	\begin{equation*}			
		B_{11}=\begin{pmatrix}
			
			1&1&1&1&1&1\\
			1&1&0&0&0&2\\
			1&0&1&2&2&a\\
			1&0&2&1&2&b\\
			1&0&2&2&1&c\\
			1&2&a&b&c&1		
		\end{pmatrix},
	\end{equation*}
	where $a,b,c\in \{0,2\}$. By direct calculation, we have $P(B_{11},0)=2(ab+ac+bc-a^2-b^2-c^2-a-b-c+2)\neq0$, then $m_{B_{11}}(0)=0$, and thus $rank(\mathcal{A}(G)+I)\geq rank(B_{11})=6-m_{B_{11}}(0)= 6$, which implies $m(-1)=m_{\mathcal{A}(G)+I}(0)=n-rank(\mathcal{A}(G)+I)\leq  n-6$, a contradiction with $m(-1)=n-5$.
	
	\textbf{Subcase 2.1.2}: $P_4$ is an induced subgraph of $G[V_2]$, and  $G[V_2]$ is $K_{1,3}$-free.
	
	Let $P_4=v_2v_3v_4v_5$. Then the principal submatrix of $\mathcal{A}(G)+I$ indexed by $\{v_1,v_2,v_3,v_4,v_5,v_6\}$ is
	\begin{equation*}			
		B_{12}=\begin{pmatrix}
			
			1&1&1&1&1&1\\
			1&1&0&2&2&e\\
			1&0&1&0&2&f\\
			1&2&0&1&0&g\\
			1&2&2&0&1&h\\
			1&e&f&g&h&1		
		\end{pmatrix},
	\end{equation*}
	where $e,f,g,h\in \{0,2\}$. By direct calculation, we have $P(B_{12},0)=2(f^2+g^2+ef+3fg+gh+5e+5h-e^2-h^2-eg-3eh-fh-5f-5g)$.
	
		  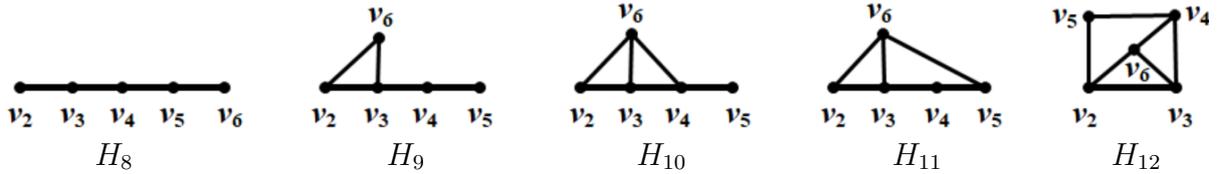
\begin{figure}[h]
		\centering
		\begin{tikzpicture}
			\node[anchor=south west,inner sep=0] (image) at (0,0) {\includegraphics[width=1\textwidth]{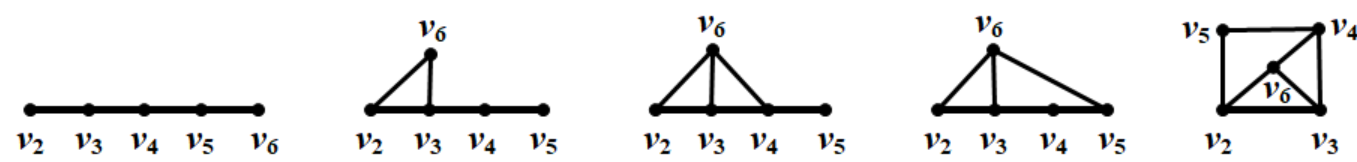}};
			\begin{scope}[
				x={(image.south east)},
				y={(image.north west)}
				]
				\node [black, font=\bfseries] at (0.1,-0.15) {$H_8$ };
				\node [black, font=\bfseries] at (0.34,-0.15) {$H_9$};
				\node [black, font=\bfseries] at (0.55,-0.15) {$H_{10}$};
				\node [black, font=\bfseries] at (0.76,-0.15) {$H_{11}$};
				\node [black, font=\bfseries] at (0.94,-0.15) {$H_{12}$};
			\end{scope}
		\end{tikzpicture}
		
		\caption{Graphs $H_8,H_9,H_{10},H_{11}, H_{12}$.}\label{Figure 3}
	\end{figure}

	 Now we show $G[V_2]\cong C_5$ or $G[V_2]\cong H_1$. Otherwise, $G[V_2]$ must be one of $H_8,H_9,H_{10},H_{11}$, which is shown in Figure 3 by $G[V_2]$ is $K_{1,3}$-free, and the values of $P(B_{12},0)$ is shown in Table 2.

\begin{table}[h]\label{12}
	\centering
	
	\renewcommand{\arraystretch}{1.5}
	
	\begin{tabular}{|p{1.5cm}|p{1cm}|p{1cm}|p{1cm}|p{1cm}|c|}
		\hline
		$G[V_2]$ & $e$& $f$& $g$& $h$&$P(B_{12},0)$ \\ \hline
		$H_8$&$2$&$2$&$2$&$0$&$12$\\ \hline
		$H_9$&$0$&$0$&$2$&$2$&$8$\\ \hline
		$H_{10}$&$0$&$0$&$0$&$2$&$12$\\ \hline
		$H_{11}$&$0$&$0$&$2$&$0$&$-12$\\ \hline
	\end{tabular}
	\caption{the values of $P(B_{12},0)$}.
\end{table}
	\vspace{-0.4cm}
	Then $m_{B_{12}}(0)=0$ by $P(B_{12},0)\neq0$, and thus $rank(\mathcal{A}(G)+I)\geq rank(B_{12})=6-m_{B_{12}}(0)= 6$, which implies $m(-1)=m_{\mathcal{A}(G)+I}(0)=n-rank(\mathcal{A}(G)+I)\leq  n-6$, a contradiction with $m(-1)=n-5$. Therefore, $G[V_2]\cong C_5$ or $G[V_2]\cong H_1$ in this subcase.
	
	\textbf{Subcase 2.1.3}: $C_4$ is an induced subgraph of $G[V_2]$, and $G[V_2]$ is $\{K_{1,3},P_4\}$-free.
	
	Without loss of generality, we assume that $G[\{v_2,v_3,v_4,v_5\}]\cong C_4$ and then  $G[V_2]\cong H_{12}$ by $G[V_2]$ is $\{K_{1,3},P_4\}$-free, where $H_{12}$ as shown in Figure \ref{Figure 3}.
	Let $B_{13}$ be the principal submatrix of $\mathcal{A}(G)+I$ indexed by $\{v_1\}\cup V_2$, where
	\begin{equation*}			
		B_{13}=\begin{pmatrix}
			
			1&1&1&1&1&1\\
			1&1&0&2&0&0\\
			1&0&1&0&2&0\\
			1&2&0&1&0&0\\
			1&0&2&0&1&2\\
			1&0&0&0&2&1		
		\end{pmatrix}.
	\end{equation*}
	Then $P(B_{13},0)=-8\neq0$ by direct calculation, and thus $rank(\mathcal{A}(G)+I)\geq rank(B_{13})=6$, which implies $m(-1)=m_{\mathcal{A}(G)+I}(0)=n-rank(\mathcal{A}(G)+I)\leq  n-6$, a contradiction.
		
		\textbf{Subcase 2.1.4}: $G[V_2]$ is $\{K_{1,3},P_4,C_4\}$-free. 
		
		Then $G[V_2]\cong K_l\vee (K_s\cup K_t)$ by Lemma \ref{fr2}, where $s,t\geq 0$ and $l\geq 1$, and thus $\bigtriangleup(G[V_2])=(l-1)+s+t=|V_2|-1=4$, a contradiction with $\bigtriangleup(G[V_2])\leq 3$.
		
		\textbf{Subcase 2.2}: $G[V_2]$ is a disconnected graph.
		
		\textbf{Subcase 2.2.1}: $G[V_2]$ has five connected components.
		
		Clearly, $G[V_2]\cong 5K_1$. Let $B_{14}$ be the principal submatrix of $\mathcal{A}(G)+I$ indexed by $\{v_1\}\cup V_2$, where
		\begin{equation*}			
			B_{14}=\begin{pmatrix}
				
				1&1&1&1&1&1\\
				1&1&2&2&2&2\\
				1&2&1&2&2&2\\
				1&2&2&1&2&2\\
				1&2&2&2&1&2\\
				1&2&2&2&2&1		
			\end{pmatrix}.
		\end{equation*}
	 Then $P(B_{14},0)=4$ by direct calculation, and thus $rank(\mathcal{A}(G)+I)\geq rank(B_{14})=6$, which implies $m(-1)=m_{\mathcal{A}(G)+I}(0)=n-rank(\mathcal{A}(G)+I)\leq  n-6$, a contradiction.
	 
	 \textbf{Subcase 2.2.2}: $G[V_2]$ has four connected components.
	 
	 Clearly, $G[V_2]\cong 3K_1\cup K_2$, where $v_2\sim v_3$. Let $B_{15}$ be the principal submatrix of $\mathcal{A}(G)+I$ indexed by $\{v_1\}\cup V_2$, where
	 \begin{equation*}			
	 	B_{15}=\begin{pmatrix}
	 		
	 		1&1&1&1&1&1\\
	 		1&1&0&2&2&2\\
	 		1&0&1&2&2&2\\
	 		1&2&2&1&2&2\\
	 		1&2&2&2&1&2\\
	 		1&2&2&2&2&1		
	 	\end{pmatrix}.
	 \end{equation*}
	  Then $P(B_{15},0)=-8$ by direct calculation, and thus $rank(\mathcal{A}(G)+I)\geq rank(B_{15})=6$, which implies $m(-1)=m_{\mathcal{A}(G)+I}(0)=n-rank(\mathcal{A}(G)+I)\leq  n-6$, a contradiction.
	 
	 \textbf{Subcase 2.2.3}: $G[V_2]$ has three connected components.
	 
	 Clearly, $G[V_2]\in\{2K_1\cup P_3, 2K_1\cup K_3,  2K_2\cup K_1\}$. 
	 
	 If $G[V_2]\cong 2K_1\cup P_3$ or $2K_1\cup K_3$, where $G[\{v_2,v_3\}]\cong 2K_1$, then the principal submatrix of $\mathcal{A}(G)+I$ indexed by $\{v_1\}\cup V_2$ is
	 \begin{equation*}			
	 	B_{16}=\begin{pmatrix}
	 		
	 		1&1&1&1&1&1\\
	 		1&1&2&2&2&2\\
	 		1&2&1&2&2&2\\
	 		1&2&2&1&0&a_1\\
	 		1&2&2&0&1&0\\
	 		1&2&2&a_1&0&1		
	 	\end{pmatrix},
	 \end{equation*}
	 where $a_1\in \{0,2\}$.
	 By direct calculation, we have $P(B_{16},0)=-2(a_1^2+3a_1-4)\neq 0$. Thus $rank(\mathcal{A}(G)+I)\geq rank(B_{16})=6$, which implies $m(-1)=m_{\mathcal{A}(G)+I}(0)=n-rank(\mathcal{A}(G)+I)\leq  n-6$, a contradiction.
	 
	 If $G[V_2]\cong 2K_2\cup K_1$, where $v_3\sim v_4,v_5\sim v_6$, then the principal submatrix of $\mathcal{A}(G)+I$ indexed by $\{v_1\}\cup V_2$ is
	 \begin{equation*}			
	 	B_{17}=\begin{pmatrix}
	 		
	 		1&1&1&1&1&1\\
	 		1&1&2&2&2&2\\
	 		1&2&1&0&2&2\\
	 		1&2&0&1&2&2\\
	 		1&2&2&2&1&0\\
	 		1&2&2&2&0&1		
	 	\end{pmatrix}.
	 \end{equation*}
	 By direct calculation, we have $P(B_{17},0)=12$. Thus $rank(\mathcal{A}(G)+I)\geq rank(B_{17})=6$, which implies $m(-1)=m_{\mathcal{A}(G)+I}(0)=n-rank(\mathcal{A}(G)+I)\leq  n-6$, a contradiction.
	 
	 \textbf{Subcase 2.2.4}: $G[V_2]$ has two connected components.
	 
	 Clearly, we have $G[V_2]\in\{K_1\cup P_4,K_1\cup C_4, K_1\cup K_{1,3}, K_1\cup \overline{P_3\cup K_1}, K_1\cup (K_2\vee 2K_1),K_1\cup K_4, K_2\cup P_3,K_2\cup K_3\} $. Now we show $G[V_2]\cong K_1\cup P_4$. 
	 
	  If $G[V_2]\cong K_1\cup C_4$, where $C_4=v_3v_4v_5v_6v_3$; or $G[V_2]\cong K_1\cup \overline{P_3\cup K_1}$, where $v_3\sim v_4,v_3\sim v_5,v_3\sim v_6,v_4\sim v_5$, then the principal submatrix of $\mathcal{A}(G)+I$ indexed by $\{v_1,v_2,v_3,v_4,v_5,v_6\}$ is  %or $K_1\cup \overline{P_3\cup K_1}$
	 \begin{equation*}			
	 	B_{18}=\begin{pmatrix}
	 		
	 		1&1&1&1&1&1\\
	 		1&1&2&2&2&2\\
	 		1&2&1&0&a_2&0\\
	 		1&2&0&1&0&2\\
	 		1&2&a_2&0&1&b_2\\
	 		1&2&0&2&b_2&1		
	 	\end{pmatrix},
	 \end{equation*}
	 where $(a_2,b_2)=\begin{cases}
	 	(2,0),&\ \text{if}\ G[V_2]\cong K_1\cup C_4\\
	 	(0,2),&\ \text{if}\ G[V_2]\cong K_1\cup \overline{P_3\cup K_1}
	 \end{cases}$. By direct calculation, we have $P(B_{18},0)=2a_2b_2+2b^2_2+8-2a^2_2-6a_2-2b_2\neq 0$.
	 
	 If $G[V_2]\cong K_1\cup K_{1,3}$, where $v_2\sim v_3$, $v_2\sim v_4$, $v_2\sim v_5$, then the principal submatrix of $\mathcal{A}(G)+I$ indexed by $\{v_1,v_2,v_3,v_4,v_5,v_6\}$ is $B_{11}$ in Subcase 2.1.1 with $a=b=c=2$, and thus $P(B_{11},0)\neq 0$.

	 If $G[V_2]\cong  K_1\cup (K_2\vee 2K_1)$ or $K_1\cup K_4$, where $G[\{v_3,v_4,v_5,v_6\}]\cong K_2\vee 2K_1$ or $ K_4$, then the principal submatrix of $\mathcal{A}(G)+I$ indexed by $\{v_1,v_2,v_3,v_4,v_5,v_6\}$ is
	 \begin{equation*}			
	 	B_{19}=\begin{pmatrix}
	 		
	 		1&1&1&1&1&1\\
	 		1&1&2&2&2&2\\
	 		1&2&1&0&0&0\\
	 		1&2&0&1&0&0\\
	 		1&2&0&0&1&a_3\\
	 		1&2&0&0&a_3&1		
	 	\end{pmatrix},
	 \end{equation*}
	 where $a_3\in \{0,2\}$.  By direct calculation, we have $P(B_{19},0)=2(a_3^2+a_3-2)\neq 0$.
	 
	 If $G[V_2]\cong K_2\cup P_3$ or $K_2\cup K_3$, where $G[\{v_4,v_5,v_6\}]=P_3$ or $K_3$, then  the principal submatrix of $\mathcal{A}(G)+I$ indexed by $\{v_1,v_2,v_3,v_4,v_5,v_6\}$ is obtained from $B_{16}$ by replacing $2$ by $0$ in the $(2,3),(3,2)$ position, denoted by $B_{20}$, and we have   $P(B_{20},0)=2(a_1^2+a_1-2)\neq 0$ by direct calculation.
	 
	 Combining the above arguments, we get $rank(\mathcal{A}(G)+I)\geq 6$, which implies $m(-1)\leq n-6$, a contradiction. Thus $G[V_2]\cong K_1\cup P_4$ in this subcase.
	 
	 By Case 1 and Case 2, we complete the proof.
\end{proof}
	\begin{table}[h]
	\centering
	
	\renewcommand{\arraystretch}{1.5}
	
	\begin{tabular}{|c|c|}
		\hline
		graph $G$ & $P(\mathcal{A}(G),\lambda)$ \\ \hline
		$K_{n-5}\vee C_5$&$(\lambda+1)^{n-5}(\lambda^2+2\lambda-4)^{2}(\lambda-n+1)$\\ \hline
		$K_{n-4}\vee(K_1\cup P_4)$ & $(\lambda+1)^{n-5}(\lambda^2+2\lambda-4)[\lambda^3-(n-3)\lambda^2-(2n+10)\lambda+4n-36]$  \\ \hline
		$K_{n-4}\vee H_1$ & $(\lambda+1)^{n-5}(\lambda^2+2\lambda-4)[\lambda^3-(n-3)\lambda^2-(2n+2)\lambda+4n-20]$ \\ \hline
	\end{tabular}
	\caption{the characteristic polynomial of $\mathcal{A}(G)$}.
\end{table}
\vspace{-0.5cm}
\begin{theorem}\label{n-53}
		Let $G$ be a connected graph with order $n\geq 16$, $\mathcal{G}_1=\{K_{n-4}\vee4K_1,K_{n-4}\vee(2K_1\cup K_2),K_{n-4}\vee(P_3\cup K_1),K_{n-4}\vee2K_2,K_{n-4}\vee P_4,K_{n-4}\vee(K_3\cup K_1),K_{n-4}\vee C_4\}$. Then $m(-1)=n-5$ if and only if $G\in \mathcal{G}_1\cup\{K_{n-5}\vee C_5,K_{n-5}\vee(K_1\cup P_4),K_{n-5}\vee H_1\}$. In addition, $G$ is determined by its $\mathcal{A}$-spectrum.
\end{theorem}
\begin{proof}
	If $m(-1)= n-5$, then $G\in \mathcal{G}_1\cup\{K_{n-5}\vee C_5,K_{n-5}\vee(K_1\cup P_4),K_{n-5}\vee H_1\}$ by Lemma \ref{n-52}.
	
	 Conversely, if $G\in \{K_{n-5}\vee C_5,K_{n-5}\vee(K_1\cup P_4),K_{n-5}\vee H_1\}$, then by direct calculation and Lemma \ref{qjz}, we have $P(\mathcal{A}(G),\lambda)$ as shown in  Table 3. Thus $m(-1)= n-5$ and $G$ is determined by its $\mathcal{A}$-spectrum if $G\in \mathcal{G}_1\cup\{K_{n-5}\vee C_5,K_{n-5}\vee(K_1\cup P_4),K_{n-5}\vee H_1\}$ by Table 1 and Table 3.
\end{proof}
\vskip0.15cm

By Theorems \ref{n-1}, \ref{n-2}, \ref{n-32}, \ref{n-42}, \ref{n-53}, the results of Theorem \ref{1} hold, and we complete the proof. Clearly, if there exists $v\in V(G)$ such that $\varepsilon(v)=1$, then $\mathcal{A}(G)$ is irreducible, and thus we have the following conclusion by Theorem \ref{1} immediately.

\begin{remark}\label{r2}
	Let $G$ be a connected graph with order $n\geq16$ and $m(-1)\geq n-5$. Then $\mathcal{A}(G)$ is irreducible.
\end{remark}

\section*{\bf Conclusion}

\hspace{1.5em} In this paper, by using equitable quotient matrix, interlacing theorem, the relationship between rank of matrix and the multiplicity of eigenvalue $0$ etc., we characterize graphs with $m(-1)= n-i$, where $i\leq 5$. Actually, we can also characterize graphs  with $m(-1)= n-6$ in a similar way, but the calculation would be complicated.

Fowler and Pisanski \cite{fo} introduced  HL-index of the adjacency matrix of a graph. It is related to the HOMO-LUMO separation studied in the theoretical chemistry. Similarly, Wang et al.\cite{wang3} defined the HL-index $R_\mathcal{A}(G)=\max\{|\xi_{H}|,|\xi_{L}|\}$ with respect to the eccentricity matrix $\mathcal{A}(G)$ of a connected graph $G$ with order $n$, where $H=\lfloor \frac{n+1}{2}\rfloor$, $L=\lceil \frac{n+1}{2} \rceil$. Usually, $\xi_{H}$ and $\xi_{L}$ are called median eigenvalues. Clearly, for sufficiently large $n$, we always have $R_\mathcal{A}(G)=1$ for a connected graph $G$ with $m(-1)= n-i$, where $i\leq 5$.

After careful consideration and demonstration, we propose the following conjecture and two problems about the multiplicity of $-1$ and $0$ in the spectrum of the eccentricity matrix for further research.

\textbf{Conjecture 1}. Let $G$ be a connected graph with order $n$, $m(\xi)=n-i$ for $1\leq i\leq n-\lceil \frac{n+1}{2}\rceil$. Then $\xi\in\{-2,-1,0\}$ for a large enough $n$.

\textbf{Problem 1}. Determine the connected graphs with order $n$ and $m(-1)= n-i$, where $6\leq i\leq n-\lceil \frac{n+1}{2}\rceil$.

\textbf{Problem 2}. Determine the connected graphs with order $n$ and $m(0)= n-i$, where $i\leq 4$.

\section*{\bf Funding}

\hspace{1.5em} This work is  supported by the National Natural Science Foundation of China (Grant Nos. 12371347, 12271337).


\begin{thebibliography}{99}


\bibitem{bd}J.A. Bondy, U.S.R. Murty, Graph Theory, Spring, London, 2008.

\bibitem{AE} A.E. Brouwer, W.H. Haemers, Spectra of Graphs, Springer, New York, 2012.
%\bibitem{kcd}K.C. Das, S.W. Sun, Normalized Laplacian eigenvalues and energy of trees, Taiwan. J. Math. 20 (3) (2016b) 491–507.

%\bibitem{kcd2}K.C. Das, S.W. Sun, Extremal graph on normalized Laplacian spectral radius and energy. Linear Algebra appl., 29(2016)237-253.

\bibitem{is} D.M. Cvetkovi\'{c}, M. Doob, H. Sachs, Spectra of Graphs: Theory and Applications, 1980.

%\bibitem{lj}J.X. Li, J.M. Guo, W.C.Shiu, Bounds on normalized Laplacian eigenvalues of graphs, J. Inequal. Appl. 316 (2014) 1–8.

%\bibitem{cm} M. Cavers, The Normalized Laplacian Matrix and General Randi\'{c} Index of Graphs (Ph. D. dissertation), University of Regina, 2010.

%\bibitem{sun-s}S.W. Sun, K.C. Das, On the second largest normalized Laplacian eigenvalue of
%graphs, Applied Mathematics and Computation, 348 (2019) 531-541.



\bibitem{dc}D. Cvetkovi\'{c}, P. Rowlinson, S.K. Simi\'{c}, An Introduction to the Theory of Graph Spectra, Cambridge Univ. Press, Cambridge, 2010.

\bibitem{fo} P.W Fowler, T. Pisanski, HoMO-LUMO maps for fullerenes, Acta Chimica Slovenica, 57 (2010) 513-517.

\bibitem{gx} X. Gao, Z. Stani\'{c}, J.F. Wang, Grahps with large multiplicity of $-2$ in the spectrum of the eccentricity matrix,  Discrete Mathematics, 347 (2024) 114038.	



\bibitem{lei} X.Y. Lei, J.F. Wang, G.Z. Li, On the eigenvalues of eccentricity matrix of graphs, Discrete Applied Mathematics, 295 (2021) 134-147.

\bibitem{hl} H.Q. Lin, M.Q. Zhai, S.C. Gong, On graphs with at least three distance eigenvalues less than $-1$, Linear Algebra and its Application, 458 (2014) 548-558.

\bibitem{vz} V. Lozin, Graph theory notes, 2017, https://homepages.warwick.ac.uk/masgax/graph-Theory-notes. Accessed 27 September 2017.

\bibitem{mahat} I. Mahato, M.R. Kannan, On the eccentricity matrix of trees: Inertia and spectral symmetry, Discrete Mathematics, 345 (2022) 113067.	

\bibitem{mo} B. Mohar, Median eigenvalues and the HOMO-LUMO index of graphs, Journal of Combinatorial Theory Series B, 112 (2015) 78-92.


\bibitem{mo2} B. Mohar, Median eigenvalues of bipartite subcubic graphs, Combinatorics Probability \& Computing, 25 (2016) 768-790.

\bibitem{zpq} Z.P. Qiu, Z.K. Tang, Q.Y. Li, Eccentricity spectral radius of t-clique trees with given diameter. Discrete Applied Mathematics, 337 (2023) 202-217.

\bibitem{ss} S. Sorgun, H. H\"{u}\c{c}\"{u}k, On two problems related to anti-adjacency (eccentricity) matrix, Discrete Applied Mathematics, 328 (2023) 1-9.





\bibitem{wang3} J.F. Wang, X.Y. Lei, M. Lu, S. Sorgun, et al., On graphs with exactly one anti-adjacency eigenvalue and beyond, Discrete Mathematics, 346 (2023) 113373.	


\bibitem{wang} J.F. Wang, X.Y. Lei, W. Wei, X.B. Luo, S.C. Li, On the eccentricity matrix of graphs and its applications to the boiling point
of hydrocarbons, Chemometrics and Intelligent Laboratory Systems, 207 (2020) 104173.

\bibitem{wang2} J.F. Wang, M. Lu, F. Belardo, M. Randi\'{c}, The anti-adjacency matrix of a graph: eccentricity matrix, Discrete Applied Mathematics, 251 (2018) 299–309.

\bibitem{wei} W. Wei, X.C. He, S.C. Li, Solutions for two conjectures on the eigenvalues of the eccentricity matrix, and beyond, Discrete Mathematics, 343 (2020) 111925.

\bibitem{wei2} W. Wei, S.C. Li,  On the eccentricity spectra of complete multipartite graphs, Applied Mathematics and Computation, 424 (2022) 127036.




\bibitem{ylh} L.H. You, M. Yang, W. So, W.G. Xi, On the spectrum of an equitable quotient matrix and its application, Linear Algebra and its Application, 577 (2019) 21-40.






\end{thebibliography}
\end{document}